\documentclass[12pt,twoside]{amsart}
\usepackage[dvips]{graphics}
\usepackage[all]{xy}
\usepackage{a4,amsmath,amssymb,amsfonts,amscd,mathrsfs}
\reversemarginpar
\newtheorem{theorem}[subsection]{Theorem}
\newtheorem{lemma}[subsection]{Lemma}
\newtheorem{corollary}[subsection]{Corollary}
\newtheorem{conjecture}[subsection]{Conjecture}
\newtheorem{proposition}[subsection]{Proposition}

\newcommand{\curleq}{\preccurlyeq}

\newcommand{\diag}{{\rm diag}}

\newcommand{\id}{{\rm id}}
\newcommand{\inv}{^{-1}}
\newcommand{\matvier}[1]{\left(\begin{array}{rrrr}#1\end{array} \right)}
\renewcommand{\phi}{\varphi}

\newcommand{\sca}[2]{\left\langle #1, #2 \right\rangle}
\newcommand{\spa}{{\rm span}}
\newcommand{\vect}[1]{\left(\begin{array}{r}#1\end{array} \right)}

\newcommand{\cdop}{{\mathbb C}}
\newcommand{\edop}{{\mathbb E}}
\newcommand{\fdop}{{\mathbb F}}
\newcommand{\ndop}{{\mathbb N}}

\newcommand{\rdop}{{\mathbb R}}
\newcommand{\zdop}{{\mathbb Z}}
\newcommand{\Bcal}{{\mathcal B}}
\author{Juan Marcos Cervi\~no}
\address{FB Mathematik, University Duisburg-Essen,
45117 Essen, Germany}
\email{juan.cervino@uni-due.de}
\author{Georg Hein}
\address{FB Mathematik, University Duisburg-Essen,
45117 Essen, Germany}
\email{georg.hein@uni-due.de}
\date{October 12, 2009}
\begin{document}
\title{The Conway-Sloane tetralattice pairs are non-isometric}
\begin{abstract}
Conway and Sloane constructed a $4$-parameter family of pairs of
isospectral lattices of rank four. They conjectured that all pairs in
their family are non-isometric, whenever the parameters are pairwise
different, and verified this for classical integral lattices 
of determinant up to $10^4$. In this paper, we use our theory of
lattice invariants developed in \cite{CH1} and \cite{CH2} to
prove this conjecture.
\end{abstract}
\maketitle
\section{Introduction}

The isometry classes of unary, binary and ternary positive definite
quadratic forms are determined by the representation numbers. That this
fact does not hold in any dimension, was shown by E.~Witt's example of
two non-isometric, positive definite quadratic forms in dimension
$16$ with the same representation numbers.

\vspace*{.5em}
If two positive definite quadratic forms have the same representation
numbers, then we call them {\sl isospectral}. A.~Schiemann conducted a
computer search to provide an
example of two isospectral positive definite quaternary quadratic forms
with integer coefficients which are not isometric (see \cite{Sch}).
Hence, already in rank 4, the {\em theta series}, which is the
generating series for the representation numbers, does not determine the
isometry class.

\vspace*{.5em}
In \cite{CS}, Conway and Sloane introduced a real $4$-parameter family of
pairs of isospectral lattices in the euclidean space $\mathbb{E}^4$, where Schiemann's
example is a member of. They conjectured
that the lattice pairs are non-isometric
whenever the parameter coordinates are pairwise different.
They verified this for lattice pairs
corresponding to classical integral quadratic forms of
discriminant less than $10^4$.

\vspace*{.5em}
In this article we prove the conjecture of Conway and Sloane using
our theory of lattice invariants introduced in \cite{CH1} and
\cite{CH2}. More precisely, for each tuple $(m_1,\ldots,m_k)$ of natural numbers, we associate
in \cite{CH2} a lattice invariant $\Theta_{m_1,\ldots,m_k}$.
It is an analytic function on the upper half plane, which gives a
modular form for integral lattices.
For example, $\Theta_0$ is the classical theta series of the lattice.
In \cite[Proposition 4.4]{CH1}, we showed that for Schiemann's example
the invariants $\Theta_{1,1}$ are different, hence they are not
isometric.

\vspace*{.5em}
One observes that
the function $\Theta_{1,1}$ is analytic in the four parameter
coordinates of the Conway-Sloane family.
This implies the Conway-Sloane conjecture on a dense open
subset of the parameter domain. Motivated by this observation, we
started a thorough investigation of the invariant
$\Theta_{1,1}$ for the lattice pairs in the Conway-Sloane
family. We show that for each pair the functions $\Theta_{1,1}$ are not
equal, provided that the parameter coordinates are
pairwise different -- and so proving the full conjecture of Conway and Sloane in
\cite[Remark (v)]{CS}.

\vspace*{.5em}
The invariant $\Theta_{1,1}$ enables us to give the first example of
non-isometric, isospectral lattices varying in a continuous family. So
far, there were used only ad-hoc methods for proving non-isometry of
isospectral lattices -- which usually can not be extended to such
families with real parameters.

\vspace*{.5em}
In Section \ref{TETRA} we start with an alternative description of the
lattice pair $(L_1,L_2)$ of Conway and Sloane. 
For this, we use an action of the
Kleinian four group on the self-dual codes in $\fdop_3^4$. This
construction explains the term {\em tetralattice}, as already introduced
in \cite{Con}. We repeat
the definition of the invariants $\Theta_{1,1}(\tau,L_i)$ in Section
\ref{DELTA}. Furthermore, we develop an explicit formula for the
$q$-expansion of
$\delta(\tau)=\frac{1}{128}(\Theta_{1,1}(\tau,L_1)-\Theta_{1,1}(\tau,L_2))$.
In the next section we determine those vectors contributing to the first
coefficient of the $q$-expansion of $\delta$.
Finally, we prove our main result, Theorem \ref{main}, by computing this
coefficient which turns out to be negative.
Using our lattice invariant $\Theta_{1,1}$ this result reduces, in the
end, to a simple computation.

\subsection*{Notation}
In this article, $\mathbb{E}^n$ denotes the euclidean $n$-dimensional vector space
with inner product $\langle\cdot,\cdot\rangle$. For any $v\in\mathbb{E}^n$,
$\|v\|^2 =\langle v,v\rangle$ is called the {\sl square norm} of $v$.

\section{The isospectral family of Conway and Sloane}\label{TETRA}
\subsection{A lattice with an action of the Kleinian group $K_4$}
We start with a lattice $L\cong \zdop^4$ together with its Gram
matrix
\[ G_L = \matvier{r & \alpha & \beta & \gamma\\\alpha & r & -\gamma &
-\beta\\\beta&-\gamma&r&-\alpha\\\gamma&-\beta&-\alpha&r} \, . \] 
We see that the Kleinian four group $K_4$
acts on $L$ as isometries when given as:
\[ K_4= \left\{ g_0=\id,
g_1=\matvier{0&0&1&0\\0&0&0&-1\\1&0&0&0\\0&-1&0&0}, \,
g_2=\matvier{0&1&0&0\\1&0&0&0\\0&0&0&-1\\0&0&-1&0}, \, g_3=g_2 \cdot g_1
\right\} . \]

\subsection{Sublattices of $L$ from ternary codes}\label{sub-lat}
Using the above identification $L \cong \zdop^4$ we obtain an
isomorphism $L/3L \cong \fdop_3^4$, and a surjection $\pi:L \to
\fdop_3^4$.
For each linear subspace $C \subset \fdop^4_3$ we obtain a sublattice
$L_C:=\pi\inv(C)$ of $L$ containing $3L$. 
Linear subspaces of $\fdop^4_3$ are called ternary codes. When we
speak of a code $C$, we always mean a code $C \subset \fdop^4_3$.
Since the above action of $K_4$ on $L$ maps $3L$ to $3L$, we obtain an
action of $K_4$ on $\fdop_3^4$.

If two linear codes $C$ and $C'$ differ by an element $g \in K_4$,
that is $C=g(C')$, then $L_{C'}$ and $L_C$ are isometric because the
elements of $K_4$ are isometries. On $\fdop^4_3$ we consider the non
degenerate standard scalar product $\langle\,,\rangle: \fdop^4_3 \times \fdop^4_3
\to \fdop_3$. One easily verifies that the action of $K_4$ on
$\fdop_3^4$ preserves this bilinear form.
A code $C$ is called self-dual when $C$ is of dimension 2, and
$\sca{c}{c'}=0$ for all $c,c' \in C$.
A straightforward calculation shows that there are exactly eight
self-dual codes. Here is the complete list:
\[\begin{array}{ll}
C_1  =  \spa\{ (1,0,-1,-1)^t,(0,1,+1,-1)^t \} &
C_2  =  \spa\{ (1,0,-1,+1)^t,(0,1,+1,+1)^t \}  \\
C_3  =  \spa\{ (1,0,-1,+1)^t,(0,1,-1,-1)^t \}  &
C_4  =  \spa\{ (1,0,+1,+1)^t,(0,1,+1,-1)^t \} \\
C_5  =  \spa\{ (1,0,+1,-1)^t,(0,1,+1,+1)^t \} &
C_6  =  \spa\{ (1,0,-1,-1)^t,(0,1,-1,+1)^t \} \\
C_7  =  \spa\{ (1,0,+1,+1)^t,(0,1,-1,+1)^t \}  &
C_8  =  \spa\{ (1,0,+1,-1)^t,(0,1,-1,-1)^t \} .\\
\end{array} \]
The action of $K_4$ on the set $\{ C_i \}_{i=1 \dots 8}$ of self-dual
codes has two orbits, namely $\{C_1,C_3,C_5,C_7\}$, and
$\{C_2,C_4,C_6,C_8\}$. There is another description of the partition of
the set $\{ C_i \}_{i=1 \dots 8}$. To see it, we draw the graph $\Gamma$
with vertices the self-dual codes. We connect two vertices $C_i$ and
$C_j$ when $\dim(C_i \cap C_j)=1$. We obtain the following picture.
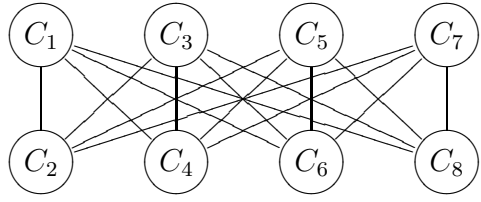
\[ \xymatrix{
*++[o][F-]{C_1} \ar@{-}[d] \ar@{-}[rd] \ar@{-}[rrd] \ar@{-}[rrrd] &
*++[o][F-]{C_3} \ar@{-}[d] \ar@{-}[rd] \ar@{-}[rrd] \ar@{-}[ld] &
*++[o][F-]{C_5} \ar@{-}[d] \ar@{-}[rd] \ar@{-}[lld] \ar@{-}[ld] &
*++[o][F-]{C_7} \ar@{-}[d] \ar@{-}[ld] \ar@{-}[lld] \ar@{-}[llld] \\
*++[o][F-]{C_2} &
*++[o][F-]{C_4} &
*++[o][F-]{C_6} &
*++[o][F-]{C_8} 
} \]
Thus, $\Gamma$ is the complete bipartite graph of type $(4,4)$. The
partition of the vertices is the above orbit partition.
\subsection{The codes $C_1$ and $C_2$}\label{c1c2}
We write down the codes $C_1$, and $C_2$ explicitly as
\[ C_1=\{0, \pm [v_0], \pm [v_1], \pm [v_2], \pm [v_3]\} , \mbox{ and }
C_2=\{0, \pm [w_0], \pm [w_1], \pm [w_2], \pm [w_3]\} \]
with
\[
v_0 = \matvier{1\\-1\\1\\0}, \,
v_1 = \matvier{0\\1\\1\\-1}, \,
v_2 = \matvier{-1\\0\\1\\1}, \,
v_3 = \matvier{-1\\-1\\0\\-1}, \, \mbox{ and }
\]
\[
w_0 = \matvier{1\\-1\\1\\0}, \,
w_1 = \matvier{1\\1\\0\\-1}, \,
w_2 = \matvier{0\\-1\\-1\\-1}, \,
w_3 = \matvier{1\\0\\-1\\1}.  \qquad 
\]
We observe that for each $v \in C_1$ different from zero there exists
exactly one $g \in K_4$ such that $g(v) \in C_2$. We arranged the
notation in such a way that $g_i(v_i)=w_i$, and $g_i(w_i)=v_i$ for all
$i=0,\ldots,3$.
\subsection{The isospectral lattices $L_1$ and $L_2$}\label{iso-spec}
We obtain two lattices $L_1=\pi\inv(C_1)$ and $L_2=\pi\inv(C_2)$.
Both are sublattices of $L$ of index 9 which contain $3L$. We show that
$L_1$ and $L_2$ have the same length spectra.
Any vector $l \in L_1$ has a unique form $l=3l_1 + c_1$ with $l_1 \in L$
and $c_1 \in C$. Using this decomposition we give a map $\Psi:L_1 \to
L_2$ by
\[ \Psi(3l_1) = 3l_1 \mbox{ , and } \Psi(3l_1 \pm v_i) = g_i(3l_1 \pm
v_i) = 3g_i(l_1) \pm w_i \, . \] 
It is easy to write down the inverse $\Phi:L_2 \to L_1$ of $\Psi$
following the same recipe:
\[ \Phi(3l_2) = 3l_2 \mbox{ , and } \Phi(3l_2 \pm w_i) = g_i(3l_2 \pm
w_i) = 3g_i(l_2) \pm v_i . \]
Since $K_4$ acts by isometries the lengths of $l \in L_1$ and
$\Psi(l) \in L_2$ coincide.
The bijection $\Psi$ is not linear.

\subsection{A new basis}
We consider the four vectors 
\[
u_0 =  \frac{1}{4}  \matvier{-1\\1\\1\\1}, \,
u_1 =  \frac{1}{4}\matvier{1\\-1\\1\\1}, \,
u_2 =  \frac{1}{4}\matvier{1\\1\\-1\\1}, \,
u_3 =  \frac{1}{4}\matvier{1\\1\\1\\-1}.
\]
These are common eigenvectors for the action of $K_4$ on $\rdop^4=\rdop
\otimes L$.
Indeed, with respect to this basis the action of $g_1$ is given by the
diagonal matrix $\diag(-1,1,-1,1)$, and the action of $g_2$
corresponds to $\diag(-1,-1,1,1)$. The Gram matrix with respect to
$\Bcal=\{u_0,u_1,u_2,u_3\}$  is given by
\[G_\Bcal =
\matvier{a&0&0&0\\
0&b&0&0\\
0&0&c&0\\
0&0&0&d}
\quad \mbox{ with } \quad
\begin{array}{rcl}
a&=& \frac{1}{4} (r-\alpha-\beta-\gamma),\\
b&=& \frac{1}{4} (r-\alpha+\beta+\gamma),\\
c&=& \frac{1}{4} (r+\alpha-\beta+\gamma),\\
d&=& \frac{1}{4} (r+\alpha+\beta-\gamma).\\
\end{array} \]
Taking as lattice basis of $L$ the column vectors of the matrix
\[ \left( \begin{array}{rrrr} 1&0&1&1\\-1&1&1&0\\1&1&0&1\\
0&-1&-1&-1 \end{array} \right) \]
with respect to the standard basis. We obtain as generators
with respect to the basis $\Bcal$ the column vectors of
\[ \left( \begin{array}{rrrr}
-1&1&-1&-1\\3&-1&-1&1\\-1&-1&1&-1\\1&3&3&3\end{array} \right) .\]

Denoting these lattice vectors by $l_0$, $l_1$, $l_2$ and $l_3$, then
$L_1$ is given by $L_1=\spa\{ l_0,l_1,3l_2,3l_3 \}$, and $L_2$ can be
described as $L_2=\spa\{ l_0,3l_1,l_2,3l_3 \}$.
From this description it is obvious that both lattices contain the
lattice $L_{12}=L_1 \cap L_2 = \spa\{ l_0,3l_1,3l_2,3l_3 \}$ as a sublattices
of index three.
\subsection{Conway and Sloane's description of $L_1$ and
$L_2$}\label{CoSl}
Performing elementary operations with column vectors, we see
that $L_2$ is generated by the columns of the matrix
\[ \matvier{-3&1&1&1\\-1&-3&-1&1\\-1&1&-3&-1\\-1&-1&1&-3\\} \]
with respect to the basis $\Bcal$.
This is the original definition of the lattice $L^-$ in \cite{CS}.
For $L_1$ we find that its lattice generators with respect to $\Bcal$
are the columns of the matrix
\[ \matvier{3&1&1&1\\1&-3&1&-1\\-1&1&3&-1\\-1&-1&1&3\\} . \]
Up to the diagonal matrix $\diag(1,-1,1,1)$ which is an isometry with respect
to the orthogonal basis $\Bcal$ this gives the
lattice $L^+$ in \cite{CS}.
We prefer the presented form to the one of Conway and Sloane.
In our form both lattices contain the same index nine lattice $M=3L$
spanned by the four vectors
\[ 
m_0=\vect{-3\\3\\3\\3}, \, 
m_1=\vect{3\\-3\\3\\3} , \, 
m_2=\vect{3\\3\\-3\\3} , \mbox{ and }
m_3=\vect{3\\3\\3\\-3} . \]
\subsection{The conjecture of Conway and Sloane}
The lattices $L_1$ and $L_2$ (respectively $L^+$ and $L^-$) depend on the
real numbers $a$, $b$, $c$, and $d$. To express this dependence we write
$L_{1;a,b,c,d}$ and $L_{2;a,b,c,d}$. Considering a large (but finite)
number of examples Conway and Sloane formulated the following
\begin{conjecture} 
For all real numbers $(a,b,c,d) \in \rdop^4$ subject to the condition
$0 <a<b<c<d$ the lattices $L_{1;a,b,c,d}$ and $L_{2;a,b,c,d}$ are isospectral
but not isomorphic.
\end{conjecture}
{\em Remark 1.}
The above conjecture is a generalization of an example found
by Schiemann in \cite{Sch}.
His example is the case $(a,b,c,d)=(1,7,13,19)$.

{\em Remark 2.}
It was shown by Conway and Sloane in \cite{CS}
(and above in in \ref{iso-spec}) that $L_1$ and $L_2$ are isospectral.

{\em Remark 3.}
As mentioned in \cite[Remark (ii)]{CS} the condition $0<a<b<c<d$ may be
replaced by: $(a,b,c,d) \in \rdop^4_+$ and the four numbers are pairwise
different. 

\section{The discrepancy of a lattice pair}\label{DELTA}
We will distinguish $L_1$ and $L_2$ using our invariant $\Theta_{1,1}$
introduced in \cite{CH1}.  We briefly review its definition and
$q$-expansion.  The discrepancy $\delta$ of the lattice pair $(L_1,L_2)$
is defined to be the difference $2^{-7}\left(\Theta_{1,1}(L_1) -
\Theta_{1,1}(L_2)\right)$.
We develop the $q$-expansion for the discrepancy. 

\subsection{The invariant $\Theta_{1,1,L}$}
For a lattice $L \subset \edop^n$ in the $n$ dimensional Euclidean
space $\edop^n$, and a polynomial $h: \edop^n \to \cdop$ we
denote by $\Theta_{h,L}$ the weighted theta function
\[ \Theta_{h,L}(\tau): = \sum_{l \in L} h(l) q^{\|l\|^2} \mbox{ with }
q = \exp(2 \pi i \tau) .\]
This is an absolutely convergent power series for $\tau$ in the upper
half plane (cf.~\cite[Section 3.2]{Zag} and \cite[Section 6]{Elk}).
While these functions depend
on the embedding $L \subset \edop^n$, there are algebraic combinations of
them which are independent of the embedding:
\begin{theorem}\label{thm31}(cf.~\cite[Theorem 4.2]{CH1})
For a lattice $L \subset {\mathbb E}^4$, the analytic function 
\[ \Theta_{1,1,L}(\tau):=\Theta_{1,1}(\tau,L):= 32\left( \sum_{1 \leq i < j \leq 4}
\Theta_{x_ix_j,L}^2(\tau) \right) + \sum_{i=1}^4
\Theta_{4x_i^2-\sum_{j=1}^4x_j^2,L}^2(\tau) \]
is an analytic function in $\tau$ which is independent of the
embedding $L \to {\mathbb E}^4$.
The function $\Theta_{1,1,L}$ can be expressed in terms of $q =\exp(2\pi
i \tau)$.
Its $q$-expansion is given by 
\[\Theta_{1,1}(\tau,L) = \sum_{m \geq 0}a_m q^m \quad with \quad
a_m=
4  \! \! \! \! \! \! \! \! \! \! \! \!
\sum_{\tiny \begin{array}{c}
(l,k) \in L \times L\\
\|l\|^2+\|k \|^2=m\\
\end{array}} \! \! \! \! \! \! \! \! \! \! \left(4 \cos^2(\measuredangle
(l,k ))-1\right)\|l\|^2\|k \|^2. \]
\end{theorem}

\begin{proof} The defining equation gives $\Theta_{1,1}$ as a finite sum of
products of analytic functions. Therefore $\Theta_{1,1}$ itself is
analytic. It follows immediately from the second equality that
$\Theta_{1,1}$ is independent of the chosen embedding.
To show the equivalence of both expressions is a straightforward calculation:
\[ \begin{array}{rcl}
\Theta_{1,1}(\tau,L)
& =& \! \! \!
\sum\limits_{(l,k) \in
L \times L}\left( 32 \sum\limits_{1 \leq i < j \leq 4}
l_il_jk_ik_j  + \sum\limits_{i=1}^4
(4l_i^2-\|l\|^2)(4k_i^{2}-\|k \|^2)
\right)q^{\|l\|^2+\|k\|^2}\\
& =& \! \! \!
\sum\limits_{(l,k) \in
L \times L}\left( 16 \sum\limits_{i=1}^4
\sum\limits_{j=1}^4
l_il_jk_ik_j   -4 \|l\|^2\|
k\|^2 \right)q^{\|l\|^2+\|k\|^2}\\
& =& \! \! \! 
\sum\limits_{(l,k) \in
L \times L}\left( 16 
\langle l, k \rangle^2  -4 \|l\|^2\|
k\|^2 \right)q^{\|l\|^2+\|k\|^2} \, .\\
\end{array}\]
Now the definition of the cosine gives the formula for the
$q$-expansion.
\end{proof}

\subsection{The analytic function $\delta$}
We define the analytic function $\delta$ to be |up to a scaling
factor| the difference of the two lattice invariants $\Theta_{1,1,L^+}$ and
$\Theta_{1,1,L^-}$:
\[ \delta(\tau,a,b,c,d) : = \frac{1}{128} \left(
\Theta_{1,1}(\tau,L_{1;a,b,c,d}) -
\Theta_{1,1}(\tau,L_{2;a,b,c,d}) \right) .\]
Even though the four real parameters $(a,b,c,d)$ are part of the
definition we usually omit them for brevity.
\begin{lemma}\label{sca-for}
We have the $q$-expansion
\[ \delta(\tau) = \frac{1}{8}\sum _{(l,k) \in L_1 \times L_1 }
\left( \sca{l}{k}^2 -\sca{\Psi(l)}{\Psi(k)}^2 \right)
q^{\|k\|^2+\|l\|^2} . \]
\end{lemma}
\begin{proof}
We have seen in the proof of Theorem \ref{thm31} that 
\[ \Theta_{1,1,L_1}(\tau) = \sum\limits_{(l,k) \in
L_1 \times L_1}\left( 16 
\langle l, k \rangle^2  -4 \|l\|^2\|
k\|^2 \right)q^{\|l\|^2+\|k\|^2} \, .\]
Using the length preserving bijection $\Psi:L_1 \to L_2$ from
\ref{iso-spec} we can write 
\[\begin{array}{rcl} \Theta_{1,1,L_2}(\tau)& =& \displaystyle
\sum\limits_{(l,k) \in
L_1 \times L_1}\left( 16 
\langle \Psi(l), \Psi(k) \rangle^2  -4 \|\Psi(l)\|^2\|
\Psi(k)\|^2 \right)q^{\|\Psi(l)\|^2+\|\Psi(k)\|^2}\\
& =& \displaystyle
\sum\limits_{(l,k) \in
L_1 \times L_1}\left( 16 
\langle \Psi(l), \Psi(k) \rangle^2  -4 \|l\|^2\|
k\|^2 \right)q^{\|l\|^2+\|k\|^2} . \\
\end{array}
\]
Now the definition of $\delta$ implies the stated formula.
\end{proof}

Next we define for $[v],[v'] \in L_1/M$ the analytic functions
$\delta_{[v],[v']}$ by
\[ \delta_{[v],[v']}(\tau) = \sum\limits_{(l,k) \in
[v] \times [v']}\left( \sca{l}{k}^2-
\langle \Psi(l), \Psi(k) \rangle^2 \right)q^{\|l\|^2+\|k\|^2} .
\]
Since every vector in $L_1$ lies in exactly one class of $L_1/M$ we obtain
from Lemma \ref{sca-for}
\begin{equation}
\delta(\tau) =\frac{1}{8} \sum\limits_{([v],[v']) \in L_1/M \times L_1/M}
\delta_{[v],[v']} \, .
\end{equation}
\subsection{Recalling notation}\label{total-recall}
Before we proceed, we give a system of representatives for $L_1/M$. We
use the description from
\ref{c1c2} as $L_1/M= \{ [0],\pm [v_0],\pm [v_1],\pm [v_2],\pm [v_3]\}$
where
\[ v_0 = \matvier{-1\\3\\-1\\1}, \,
v_1 = \matvier{1\\-1\\-1\\3}, \,
v_2 = \matvier{3\\1\\-1\\-1}, \,
v_3 = \matvier{-1\\-1\\-3\\-1} \] with respect to the basis $\Bcal$.
Furthermore, we recall that for $m \in M$ we have $\Psi(m)=m$ and
$\Psi(m \pm v_i) =g_i(m \pm v_i)$. The $g_i$ are isometries which are
given with respect to $\Bcal$ by the diagonal matrices
$g_0 = \id$, $g_1=\diag(-1,1,-1,1)$, $g_2=\diag(-1,-1,1,1)$, and
$g_3=\diag(1,-1,-1,1)$.

\begin{lemma}\label{delta-relations}
The following relations among the $\delta_{[v],[v']}$ hold:
\begin{enumerate}
\item $\delta_{[v],[v]} = 0$ for all $[v] \in L_1/M$.
\item $\delta_{[v],[v']} = \delta_{[v'],[v]}$
for all pairs $[v],[v'] \in L_1/M$.
\item $\delta_{[v],[-v']}=\delta_{[v],[v']}$
for all pairs $[v],[v'] \in L_1/M$.
\item $\delta_{[0],[v]} = 0$ for all $[v] \in L_1/M$.
\end{enumerate}
\end{lemma} 

\begin{proof}
(1) We assume that $v=v_i$. The cases when $v=-v_i$ or $v=0$ work similar.
Now we rewrite the expression for $\delta_{[v_i],[v_i]}$ as follows
\[ \delta_{[v_i],[v_i]}(\tau) = \!\!\!\! \sum\limits_{(m,m') \in M \times M}
\left( \sca{m+v_i}{m'+v_i}^2-
\langle \Psi(m+v_i), \Psi(m'+v_i) \rangle^2
\right)q^{\|m+v_i\|^2+\|m'+v_i\|^2} . \]
Since $\Psi(m+v_i)=g_i(m+v_i)$, $\Psi(m'+v_i)=g_i(m'+v_i)$, and $g_i$ is
an isometry, all summands are zero.\\
(2) follows immediately from the definition of $\delta_{[v],[v']}$ and
$\delta_{[v'],[v]}$.\\
(3) First  we expand the expression for $\delta_{[v],[-v']}$.
\[ \delta_{[v],[-v']}(\tau) = \!\!\!\! \sum\limits_{(m,m') \in M \times M}
\left( \sca{m+v}{m'-v'}^2-
\langle \Psi(m+v), \Psi(m'-v') \rangle^2
\right)q^{\|m+v\|^2+\|m'-v'\|^2} . \]
Changing the summation parameter $m'=-m''$ we obtain
\[ \delta_{[v],[-v']}(\tau) = \!\!\!\! \sum\limits_{(m,m'') \in M \times M}
\left( \sca{m+v}{-m''-v'}^2-
\langle \Psi(m+v), \Psi(-m''-v') \rangle^2
\right)q^{\|m+v\|^2+\|-m''-v'\|^2} . \]
Since we have $\Psi(-m''-v')=-\Psi(m''+v')$, we obtain
\[ \delta_{[v],[-v']}(\tau) = \!\!\!\! \sum\limits_{(m,m'') \in M \times M}
\left( \sca{m+v}{m''+v'}^2-
\langle \Psi(m+v), \Psi(m''+v') \rangle^2
\right)q^{\|m+v\|^2+\|m''+v'\|^2}. \]
This gives the equality (3).\\
(4) From (2) we see that we may assume that $v \not\in [0]$.
By (3) we may assume that $v=v_i$ for some $i \in \{ 0,1,2,3\}$.
Before we show equality (4) we consider the action of the involution
$g_i$ on $M$. The orbits of length one correspond to the invariant vectors
under $g_i$. We denote this set by $M_i^1$. The orbits of length two we
denote by $M_i^2$. We use the disjoint union
\[M = M_i^1 \cup \bigcup_{ \{ m, g_i(m)\} \in M_i^2} \{ m, g_i(m)\} \, .
\]
Now we split up the summation over $M$ into two parts due to this
decomposition:
\[ \delta_{[0],[v_i]}  =
\sum\limits_{(m,m') \in M_i^1 \times M} \alpha_{m,m'} q^{\|m\|^2+\|m'+v_i\|^2}
+ \sum\limits_{(\{m,g_i(m)\},m') \in M_i^2 \times M}
\beta_{\{m,g_i(m)\},m'} q^{\|m\|^2+\|m'+v_i\|^2}\]
where the coefficients $\alpha_{m,m'}$ and $\beta_{\{m,g_i(m)\},m'}$ are
defined by
\[\begin{array}{rcl}
\alpha_{m,m'} &=& \sca{m}{m'+v_i}^2-
\langle \Psi(m), \Psi(m'+v_i) \rangle^2 \\
\beta_{\{m,g_i(m)\},m'} &=& \alpha_{m,m'} + \alpha_{g_i(m),m'} \\
\end{array}
\]
Now we consider the coefficients $a_{m,m'}$.
\[ \begin{array}{rcll}
\alpha_{m,m'}& =& \sca{m}{m'+v_i}^2-
\langle \Psi(m), \Psi(m'+v_i) \rangle^2\\
& =& \sca{m}{m'+v_i}^2 -\sca{m}{g_i(m'+v)}^2\\
& =& \sca{g_i(m)}{g_i(m'+v_i)}^2 -\sca{g_i(m)}{g_i(g_i(m'+v))}^2
& \mbox{ since } g_i \mbox{ is an isometry}\\
& =& \sca{g_i(m)}{g_i(m'+v_i)}^2 -\sca{g_i(m)}{m'+v}^2
& \mbox{ since } g_i \mbox{ is an involution}\\
&=& - \alpha_{g_i(m),m'} \, .\\
\end{array}\]
We deduce that the coefficients $\beta_{\{m,g_i(m)\},m'}$ are all zero.
Furthermore, the coefficients $\alpha_{m,m'}$ are zero for $m =g_i(m)$.
\end{proof}

\begin{corollary}\label{delta}
The function $\delta(\tau)$ can be expressed as
$ \delta(\tau) = \displaystyle \sum_{0 \leq i < j \leq 3}
\delta_{[v_i],[v_j]}(\tau)$.
\end{corollary}
\begin{proof}
Starting with the formula of equation (1) and the set
$\{0,\pm v_0,\pm v_1, \pm v_2, \pm v_3 \}$ of representatives for
$L_1/M$ from \ref{total-recall} we get
\[\delta(\tau) = \displaystyle \frac{1}{8}
\sum_{v,v' \in \{0,\pm v_0,\pm v_1, \pm v_2, \pm v_3 \}}
\delta_{[v],[v']}(\tau) . \]
We may remove all the
summands $\delta_{[v],[v']}$ with $[v] = \pm [v']$ by part (2) and (3)
of the above lemma. Furthermore, we may remove the summands
$\delta_{[v],[0]}$ and $\delta_{[0],[v']}$ by part (4) of Lemma
\ref{delta-relations}. There remain 48 summands.
For each $i<j$ we obtain from $\delta_{[\pm v_i],[\pm v_j]}$ and
$\delta_{[\pm v_j],[\pm v_i]}$ eight times the summand
$\delta_{[v_i],[v_j]}$ by (2) and (3) of Lemma \ref{delta-relations}. 
\end{proof}

\section{Minimal vectors and minimal pairs}\label{MIN}
In this section we determine the first exponent in the $q$-expansion of
$\delta$. By Corollary \ref{delta} we have to search for the shortest
lattice vectors in the equivalence classes $[v_i]$ only.
The shortest vector in an equivalence class depends on the real parameters
$(a,b,c,d)$. A vector is called minimal, if it is the shortest for a choice of the four
parameters. It turns out that in each equivalence class there are at
most two minimal vectors.

\subsection{Minimal vectors}
The square norm of a vector $v=\sum_{i=0}^3 \lambda_i u_i$ is given by
$\|v\|^2=a \lambda_0^2 + b \lambda_1^2 + c \lambda_2^2 + d\lambda_3^2$.
We decompose the map assigning a vector $v \in L_1$ its square norm as
follows
\[ \xymatrix{L_1 \ar[rr]^-\phi \ar[drr]_-{l \mapsto \|l\|^2} && \ndop^4
\ar[d]^\sigma
& \mbox{ with } \quad
 \phi\left(\sum\limits_{i=0}^3 \lambda_i u_i\right)  =
(\lambda_0^2,\lambda_1^2,\lambda_2^2,\lambda_3^2) \quad \mbox{ and}
\\ && \rdop & 
 \sigma(n_0,n_1,n_2,n_3)  = an_0+bn_1+cn_2+dn_3 .} \\
\]
Furthermore, we define a partial ordering $\curleq$ on $\ndop^4$ by
\[ (n_0,n_1,n_2,n_3) \curleq (n'_0,n'_1,n'_2,n'_3) \iff
\sum_{i=i_0}^3 n_i \leq \sum_{i=i_0}^3 n'_i
\mbox{ for all }i_0 \in \{0,1,2,3\} .\]
As usual, we write $(n_0,n_1,n_2,n_3) \prec (n'_0,n'_1,n'_2,n'_3)$
when $(n_0,n_1,n_2,n_3) \curleq (n'_0,n'_1,n'_2,n'_3)$ but not
$(n'_0,n'_1,n'_2,n'_3) \curleq (n_0,n_1,n_2,n_3)$ hold.

\begin{lemma}\label{cer-lem}
We have $(n_0,n_1,n_2,n_3) \prec (n'_0,n'_1,n'_2,n'_3)$ if and only if
the inequality\\
$an_0+bn_1+cn_2+dn_3 < an'_0+bn'_1+cn'_2+dn'_3$ holds for all real
numbers $(a,b,c,d)$ fulfilling $0<a<b<c<d$.
\end{lemma}
\begin{proof}
This equivalence is an obvious consequence of the equality
$an_0+bn_1+cn_2+dn_3=(d-c)n_3+(c-b)(n_2+n_3)+(b-a)(n_1+n_2+n_3) +
a(n_0+n_1+n_2+n_3)$.
\end{proof}
Using the map $\phi:L_1 \to \ndop^4$,
we may extend the relation $\prec$ to the lattice $L_1$ by
defining $l \prec l' \iff \phi(l) \prec \phi(l')$.
For a subset $L' \subset L_1$ we say that $l' \in L'$ is minimal, when
there is no $l''\in L'$ with $l'' \prec l'$.

\begin{lemma}\label{min-vec}
The following table gives all the minimal vectors in the equivalence classes
$[v_i]$ for $i=0,\ldots,3$.
\[ \begin{array}{c|c|c|c|c}
\mbox{class} & [v_0]  & [v_1] & [v_2] & [v_3]\\
\hline
\begin{array}{c}
\mbox{minimal}\\
\mbox{vectors}
\end{array}
& v_0, v_4
& v_1, v_5
& v_2
& v_3, v_6
\end{array}\,\,
\mbox{ with }
v_4 = \vect{-4\\0\\2\\-2} ,
v_5 = \vect{4\\2\\2\\0},
v_6= \vect{-4\\2\\0\\2}.
\]
\end{lemma}
\begin{proof}
The proof is similar in all four cases, so we consider here only the
equivalence class $[v_0]$ leaving the remaining cases to the reader.  First we
remark that neither $v_0 \curleq v_4$ nor $v_4 \curleq v_0$ holds. So it
is enough to show for any $w \in [v_0]$ at least one of the inequalities
$v_0 \prec w$ or $v_4 \prec w$ is satisfied, unless $w \in \{v_0,v_4\}$.
We take a vector $w = v_0 + \sum_{i=0}^3 \lambda_i m_i$ with $m_i$ the
lattice generators of $M$ from \ref{CoSl} and $\lambda_i \in \zdop$.
This means
\[ w= \vect{w_1\\w_2\\w_3\\w_4} \mbox{ with }
\begin{array}{rcrr}
w_1 &=& -1 +3(-\lambda_0 +  \lambda_1 + \lambda_2 + \lambda_3)\\
w_2 &=& 3 +3(+\lambda_0 -  \lambda_1 + \lambda_2 + \lambda_3)\\
w_3 &=& -1 +3(+\lambda_0 +  \lambda_1 - \lambda_2 + \lambda_3)\\
w_4 &=& 1 +3(+\lambda_0 +  \lambda_1 + \lambda_2 - \lambda_3)&.\\
\end{array}\]
Suppose now that the inequality $v_0 \curleq w$ is not satisfied. By
definition of the relation $\curleq$ at least 
one of the following four inequalities hold:
\begin{eqnarray}
w_4^2& <& 1\\
w_3^2+w_4^2& <& 2\\
w_2^2+w_3^2+w_4^2& <& 11\\
w_1^2+w_2^2+w_3^2+w_4^2& <& 12
\end{eqnarray}
The integer $w_4$ is congruent to 1 modulo three. Thus $w_4^2 \geq 1$. This
rules out (2). By the same argument we conclude that $w_3^2 \geq 1$
which makes inequality (3) impossible.\\
Assume now that (4) is fulfilled. Since $w_3^2+w_4^2 \geq 2$, we deduce
that  $w_2^2 <9$. However, $w_3$ is an integer multiple of 3, which
implies $w_2=0$. We conclude that $\lambda_0 -  \lambda_1 + \lambda_2 +
\lambda_3 =-1$.
This way, we obtain $\lambda_1=1+\lambda_0+ \lambda_2 + \lambda_3 $.
We obtain the following equations and inequality for $w_3$ and $w_4$:
\[w_3= 2+6(\lambda_0+ \lambda_3) , \quad
w_4 =4 +6(\lambda_0+ \lambda_2) , \mbox{ and }
w_3^2+w_4^2 < 11 .
\]
Since the $\lambda_i$ are integers, we must have $w_3=2$ and $w_4=-2$.
We conclude that $\lambda_3=-\lambda_0$, and $\lambda_2=-1-\lambda_0$.
From $\lambda_1=1+\lambda_0+ \lambda_2 + \lambda_3 $, we deduce
$\lambda_1= -\lambda_0$.
This yields
$w=\left( -4-12\lambda_0, 0, 2,-2 \right)^t$. So $v_4
\curleq w$ with equality only for $w= v_4$.\\
Finally we assume that inequality (5) holds. As before, we have
$w_1^2+w_3^2+w_4^2 \geq 3$, and $w_2$ is divisible by three. So from
$w_2^2<9$ we conclude $w_2=0$. As before we get $\lambda_1=1+\lambda_0+
\lambda_2 + \lambda_3 $. This yields the equations and
inequality for $w_1$, $w_3$, and $w_4$:
\[w_1=2+6(\lambda_2+ \lambda_3), \quad
w_3= 2+6(\lambda_0+ \lambda_3) , \quad
w_4 =4 +6(\lambda_0+ \lambda_2) , \mbox{ and }
w_1^2+w_3^2+w_4^2 < 12 .
\]
Since the $\lambda_i$ are integers this implies the three equalities
\[\lambda_2+ \lambda_3=0, \quad
\lambda_0+ \lambda_3=0, \quad
\lambda_0+ \lambda_2=-1. \]
From these equalities we deduce $\lambda_3=\frac{1}{2}$. Thus
inequality (5) is never fulfilled.\\
So we have seen that all $w \in v_0$ which are not of the form
$w=(-4-12\lambda_0,0,2,-2)^t$ satisfy $v_0 \curleq w$.
All vectors $w$ of this form with $w \ne v_4$ satisfy $v_4 \prec w$.
If $v_0 \curleq w$ and $w \curleq v_0$, then the squares of the coordinates
of $w$ coincide with those of $v_0$. Thus, we have $w=(\pm 1, \pm 3, \pm
1, \pm 1)$. The only vector of this type in $[v_0]$  is $v_0$.
\end{proof}

\begin{proposition}\label{min-pair}
Suppose we have two lattice vectors $v,v' \in L_1$ such that
$v\in [v_i]$ and $v' \in [v_j]$ for $0 \leq i < j \leq 3$.
If $(v,v') \not \in \{ (v_0,v_2),(v_5,v_2) \} $,
then for all $0<a<b<c<d$ we have
\[\|v\|^2+\|v'\|^2 > \min \{\|v_0\|^2+\|v_2\|^2, \|v_5\|^2+ \|v_2\|^2
\} .\]
\end{proposition}
\begin{proof}
First we note, that by Lemma \ref{cer-lem} it is enough to show the
statement of the proposition for the minimal vectors in each class. By Lemma
\ref{min-vec} we can list all those pairs belonging to different classes
modulo $M$:
\[ \begin{array}{c|c|r}
i&j&\phi(v_i)+\phi(v_j) \\ \hline
0&1&(2,10,2,10)\\
0&2&{\bf (10,10,2,2)}\\
0&3&(2,10,10,2)\\
0&5&(17,13,5,1)\\
0&6&(17,13,1,5)\\
1&2&(10,2,2,10)\\
\end{array} \quad
\begin{array}{c|c|r}
i&j&\phi(v_i)+\phi(v_j) \\ \hline
1 & 3 &(2,2,10,10)\\
1 & 4 &(17,1,5,13)\\
1 & 6 &(17,5,1,13)\\
2 & 3 &(10,2,10,2)\\
2 & 4 &(25,1,5,5)\\
2 & 5 &{\bf (25,5,5,1)}\\
\end{array} \quad
\begin{array}{c|c|r}
i&j&\phi(v_i)+\phi(v_j) \\ \hline
2 & 6 &{(25,5,1,5)}\\
3 & 4 &(17,1,13,5)\\
3 & 5 &(17,5,13,1)\\
4 & 5 &(32,4,8,4)\\
4 & 6 &(32,4,4,8)\\
5 & 6 &(32,8,4,4)\\
\end{array}\]
There are two minimal 4-tuples  among the $\phi(v_i)+\phi(v_j)$ with respect to the relation $\curleq$.
These are the 4-tuples corresponding to the pairs
$(i,j) \in \{ (0,2),(2,5)\}$.
To see that this is a complete list of minimal pairs, we check that
for all pairs\\
$(i,j) \in \{(0,1), (0,3), (0,6),
(1,2), (1,3), (1,4), (1,6), (2,3)
, (3,4), (4,5), (4,6), (5,6) \}$ we have $(\phi(v_0)+\phi(v_2) )
\prec (\phi(v_i)+\phi(v_j) )$. And for all those pairs of indices\\
$(i,j) \in \{(0, 5), (1, 4), (2, 4), (2,6), (3, 4), (3, 5), (4, 5), (4, 6), (5,
6)\}$ we see that the inequality $(\phi(v_2)+\phi(v_5) )
\prec (\phi(v_i)+\phi(v_j) )$ is satisfied.\\
So we have a complete list of minimal vectors. By Lemma \ref{cer-lem}
the minimum is attained by a minimal pair, which implies the proposition.
\end{proof}

\section{Proof of the Conway-Sloane conjecture}
\begin{theorem}\label{main}
For all real numbers $(a,b,c,d)$ satisfying $0<a<b<c<d$ the lattices
$L^+ \cong L_{1,a,b,c,d}$ and $L^-=L_{2,a,b,c,d}$ are isospectral but not isometric.
\end{theorem}
\begin{proof}
We have seen that both lattices are isospectral in \ref{iso-spec}.
To show that they are not isomorphic it is enough by Theorem \ref{thm31}
to show that $\delta(\tau) = \frac{1}{128} \left(
\Theta_{1,1,L^+}(\tau) -\Theta_{1,1,L^-}(\tau)\right)$ is not zero.
We have seen in Corollary \ref{delta} that 
\[ \delta(\tau) = \sum_{ 0 \leq i < j \leq 3} \delta_{[v_i],[v_j]}(\tau)
\mbox{ with } 
\delta_{[v_i],[v_j]}(\tau) = \sum\limits_{(l,k) \in
[v_i] \times [v_j]}\left( \sca{l}{k}^2-
\langle \Psi(l),\Psi(k) \rangle^2 \right)q^{\|l\|^2+\|k\|^2} .
\]
By Proposition \ref{min-pair} the minimal value of $\|l\|^2 + \|k\|^2$
appearing in one of the $\delta_{[v_i],[v_j]}$ is 
$\min \{\|v_0\|^2+\|v_2\|^2, \|v_5\|^2+ \|v_2\|^2 \}$
and it can be attained only by the pairs
$(v_0,v_2)$ or $(v_2,v_5)$. Now we compute
\[ \begin{array}{rcl}
\sca{v_0}{v_2}^2- \sca{\Psi(v_0)}{\Psi(v_2)}^2 & = & -12(b-a)(d-c) \\
\sca{v_2}{v_5}^2- \sca{\Psi(v_2)}{\Psi(v_5)}^2 & = & -96a(c-b) .\\
\end{array}\]
Since both numbers are negative by our assumption, we conclude that
the coefficient of $q^{\min \{\|v_0\|^2+\|v_2\|^2, \|v_5\|^2+ \|v_2\|^2
\}}$ in $\delta(\tau)$  is negative. In particular it is not zero.
Therefore $\delta(\tau) \not\equiv 0$ which gives the result.
\end{proof}

\subsection*{Acknowledgment}
The authors thank Manuel Blickle for valuable remarks.
This work has been supported by the SFB/TR 45
``Periods, moduli spaces and arithmetic of algebraic varieties''.

\end{document}